\documentclass[11pt]{article}
\usepackage{amsmath}
\usepackage{amsfonts}
\usepackage{amssymb}
\usepackage{amsthm}
\usepackage{breqn}
\usepackage{setspace}
\usepackage{fullpage}
\usepackage{enumitem}
\usepackage{comment}
\usepackage{hyperref}
\usepackage{bbm}
\usepackage{tikz}
\usepackage{enumitem}
\usepackage{mathtools} 

\usepackage[utf8]{inputenc}
\usepackage[english]{babel}
\usepackage{mathtools}
\usetikzlibrary{patterns,arrows,decorations.pathreplacing}
\newtheorem{lemma}{Lemma}
\newtheorem{theorem}{Theorem} 
\newtheorem{corollary}{Corollary} 
\newtheorem{definition}{Definition}
\newtheorem{claim}{Claim}

\newtheorem{conjecture}{Conjecture}

\newtheorem{remark}{Remark}

\newcommand{\floor}[1]{\left\lfloor{#1}\right\rfloor}
\newcommand{\ceil}[1]{\left\lceil{#1}\right\rceil}

\newcommand{\RNum}[1]{\uppercase\expandafter{\romannumeral #1\relax}}
\DeclareMathOperator{\ex}{ex}
\DeclareMathOperator{\EX}{EX}

\usepackage{authblk}
\title{The Tur\'an Number of the Triangular Pyramid of $3$-Layers}

\author[1,2]{\hspace{1cm}Debarun Ghosh}
\author[1,2]{Ervin Gy\H{o}ri} 
\author[1,2]{Addisu Paulos}
\author[1,2]{\newline Chuanqi Xiao}
\author[3]{Oscar Zamora}
\affil[1]{Central European University, Budapest\par
\texttt{oscarz93@yahoo.es, chuanqixm@gmail.com, ghosh\textunderscore debarun@phd.ceu.edu,addisu\textunderscore 2004@yahoo.com}}
\affil[2]{Alfr\'ed R\'enyi Institute of Mathematics, Budapest \par
\texttt{gyori.ervin@renyi.mta.hu}}
\affil[3]{Universidad de Costa Rica, San Jos\'e}
\date{}
\begin{document}
\maketitle
\begin{abstract}
The Tur\'an number of a graph $H$, denoted by $\ex(n, H)$, is the maximum number of edges in an $n$-vertex graph that does not have $H$ as a subgraph. 
Let $TP_k$ be the triangular pyramid of $k$-layers.  In this paper, we determine that $\ex(n,TP_3)= \frac{1}{4}n^2+n+o(n)$ and pose a conjecture for $\ex(n,TP_4)$.

\end{abstract}
\section{Introduction}
The Tur\'an number of a graph $H$, denoted by $\text{ex}(n, H)$, is the maximum number of edges in an $n$-vertex graph that does not contain $H$ as a subgraph. Let $\text{EX}(n,H)$ denote the set of extremal graphs, i.e. the set of all $n$-vertex, $H$-free graph $G$ such that $e(G)=\text{ex}(n,H)$.

A systematic study of such type problems started after Tur\'an found and characterized $\text{EX}(n,K_{r+1})$.  The case $r=2$ was solved by Mantel in 1907. 
\begin{theorem}\cite{MAN}\label{mantel}
The maximum number of edges in an $n$-vertex triangle-free graph is $\floor{\frac{n^2}{4}}$. Furthermore, the only triangle-free graph with $\floor{\frac{n^2}{4}}$ edges is the complete  bipartite graph $K_{\floor{\frac{n}{2}}\ceil{\frac{n}{2}}}$.
\end{theorem}
The Tur\'an graph, $T_r(n)$, is an $n$-vertex complete $r$-partite graph whose parts have as equal as possible sizes.   Precisely speaking, the graph has ($n\ \text{mod}\ r$) parts of size $\lceil{n/r}\rceil$ and $r-(n\ \text{mod}\ r)$ parts of size $\lfloor{n/r}\rfloor$. Denote $e(T_r(n))$ by $t_r(n)$. Tur\'an proved the following fundamental result in the study of extremal graph theory:

\begin{theorem}\cite{turan1941external}\label{turan}
For an $n$-vertex $K_{r+1}$-free graph $G$, $$e(G)\leq t_{r}(n),$$
and equality holds if and only if $G$ is the Tur\'an graph $T_r(n)$, i.e., \\ $\ex(n,K_{r+1})=t_r(n)$ and $\EX(n,K_{r+1})=T_r(n)$.
\end{theorem}

In 1966, Erd\H{o}s, Stone, and Simonovits determined the asymptotic value of $\ex(n, H)$, where $H$ is a non-bipartite graph. 

\begin{theorem}\cite{EDR1,EDR2} \label{ess}
Let $F$ be a non-bipartite graph. Then
$$\ex(n,H)=\left(1-\frac{1}{\chi(H)-1}\right){n\choose 2}+o(n^2),$$
where $\chi(H)$ denotes the chromatic number of $H$.
\end{theorem}

\begin{definition}
The Triangular Pyramid with $k$ layers, denoted by $TP_k$, is defined as follows: Draw $k+1$ paths in layers such that the first layer is a $1$-vertex path, the second layer is a $2$-vertex path,\dots, and the $(k+1)^{st}$ layer is a $(k+1)$-vertex path. Label the vertices from left to right of the $i^{th}$ layer's path as $x_1^{i},x_2^{i},\dots,x_i^{i}$, where $i\in\{1,2,3,\dots, k+1\}$. 
The vertex set of the graph $TP_k$ is the set of all vertices of the $(k+1)$ paths. The edge set contains all the edges of the paths.  Additionally, for any two consecutive $(i-1)^{th}$ and $i^{th}$ layer, $x_r^{i-1}x_r^{i}$ and $x_r^{i-1}x_{r+1}^{i}$ are in $E(TP_k)$, where $i\in\{1,2,\dots,k+1\}$ and $1\leq r\leq i-1$ (see Figure \ref{PT3PT5}).
\end{definition}

\begin{figure}[h]
\centering
\begin{tikzpicture}[scale=0.4]
\draw[fill=black](0,0)circle(6pt);
\draw[fill=black](-2,-2)circle(6pt);
\draw[fill=black](2,-2)circle(6pt);
\draw[fill=black](-4,-4)circle(6pt);
\draw[fill=black](0,-4)circle(6pt);
\draw[fill=black](4,-4)circle(6pt);
\draw[fill=black](-6,-6)circle(6pt);
\draw[fill=black](-2,-6)circle(6pt);
\draw[fill=black](2,-6)circle(6pt);
\draw[fill=black](6,-6)circle(6pt);
\draw[thick](0,0)--(-6,-6)(2,-2)--(-2,-6)(4,-4)--(2,-6)(0,0)--(6,-6)(-2,-2)--(2,-6)(-4,-4)--(-2,-6);
\draw[thick](-2,-2)--(2,-2)(-4,-4)--(4,-4)(-6,-6)--(6,-6);
\node at (0,-9) {$TP_3$};
\node at (0,1) {$x_1^{1}$};
\node at (-3,-1.8) {$x_1^{2}$};
\node at (3,-1.8) {$x_2^{2}$};
\node at (-5,-3.8) {$x_1^{3}$};
\node at (0,-3) {$x_2^{3}$};
\node at (5,-3.8) {$x_3^{3}$};
\node at (-7,-6.2) {$x_1^{4}$};
\node at (-2,-6.8) {$x_2^{4}$};
\node at (2,-6.8) {$x_2^{4}$};
\node at (7,-6.2) {$x_4^{4}$};

\end{tikzpicture} \qquad\qquad\qquad
\begin{tikzpicture}[scale=0.3]
\draw[fill=black](0,0)circle(7pt);
\draw[fill=black](-2,-2)circle(7pt);
\draw[fill=black](2,-2)circle(7pt);
\draw[fill=black](-4,-4)circle(7pt);
\draw[fill=black](0,-4)circle(7pt);
\draw[fill=black](4,-4)circle(7pt);
\draw[fill=black](-6,-6)circle(7pt);
\draw[fill=black](-2,-6)circle(7pt);
\draw[fill=black](2,-6)circle(7pt);
\draw[fill=black](6,-6)circle(7pt);
\draw[fill=black](-8,-8)circle(7pt);
\draw[fill=black](-4,-8)circle(7pt);
\draw[fill=black](0,-8)circle(7pt);
\draw[fill=black](4,-8)circle(7pt);
\draw[fill=black](8,-8)circle(7pt);
\draw[fill=black](-10,-10)circle(7pt);
\draw[fill=black](-6,-10)circle(7pt);
\draw[fill=black](-2,-10)circle(7pt);
\draw[fill=black](2,-10)circle(7pt);
\draw[fill=black](6,-10)circle(7pt);
\draw[fill=black](10,-10)circle(7pt);
\draw[thick](0,0)--(-10,-10)(2,-2)--(-6,-10)(4,-4)--(-2,-10)(6,-6)--(2,-10)(8,-8)--(6,-10);
\draw[thick](0,0)--(10,-10)(-2,-2)--(6,-10)(-4,-4)--(2,-10)(-6,-6)--(-2,-10)(-8,-8)--(-6,-10);
\draw[thick](-2,-2)--(2,-2)(-4,-4)--(4,-4)(-6,-6)--(6,-6)(-8,-8)--(8,-8)(-10,-10)--(10,-10);
\node at (0,-13) {$TP_5$};
\end{tikzpicture}
\caption{Triangular Pyramids with $3$ and $5$ layers respectively.}
\label{PT3PT5}
\end{figure}
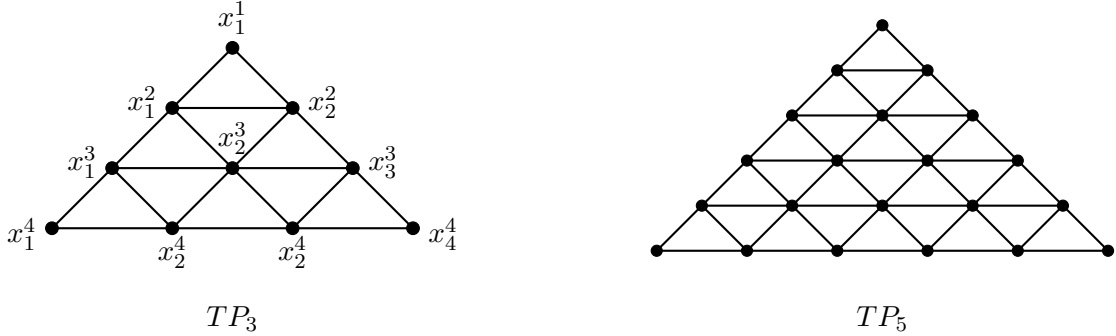

For $k\geq 1$, the chromatic number of $TP_{k}$ is $3$.  Hence by Theorem \ref{ess}, we have $\ex(n,TP_{k})=\frac{n^{2}}{4}+o(n^{2})$. Yet, it remains interesting to determine the exact value of $\ex(n,TP_{k})$. The graph $TP_1$ is a triangle and by Mantel's Theorem, $\ex(n,TP_1)=\floor{\frac{n^2}{4}}$. The graph $TP_2$ denotes the flattened tetrahedron. Liu \cite{LIU} determined  $\ex(n,TP_2)$ for sufficiently large values of $n$. Later, C. Xiao, G. O.H. Katona, J. Xiao, and O. Zamora~\cite{XIAO}  determined $\ex(n,TP_2)$ for small values of $n$.

\begin{theorem}\cite{XIAO}
The maximum number of edges in an n-vertex $TP_2$-free graph ($n\neq 5$) is,
$$ \ex(n,TP_2)=\left\{
\begin{aligned}
&\left\lfloor\frac{n^{2}}{4}\right\rfloor+\left\lfloor\frac{n}{2}\right\rfloor,& n\not\equiv2~(\bmod~4), \\
&\frac{n^{2}}{4}+\frac{n}{2}-1, & n\equiv2~(\bmod~4).
\end{aligned}
\right.
$$
\end{theorem}

In this paper, we study the Tur\'an number for $TP_{3}$, i.e. the Triangular Pyramid with three layers.

\begin{theorem}\label{maintheorem}
The maximum number of edges in an $n$-vertex $TP_3$-free graph is,
$$\ex(n,TP_3)= \frac{1}{4}n^2+n+o(n).$$
\end{theorem}

It can be checked that the constructions given in Figure \ref{fig1}, \ref{fig2} and \ref{fig3} are $TP_3$-free graphs containing $\frac{1}{4}n^2+n+1$, $\frac{1}{4}n^2+n+\frac{3}{4}$ and $\frac{1}{4}n^2+n$ edges respectively. Thus, the bound in Theorem \ref{mainthm} is best possible in terms of the linear terms, for infinitely many $n$. 

\section{Notations}
All the graphs we consider in this paper are simple and finite. Let $G$ be a graph. We denote the set of vertices and edges of $G$ by $V(G)$ and $E(G)$ respectively. The number of edges and vertices is denoted by $e(G)$ and $v(G)$ respectively.  We denote the degree of a vertex $v$ by $d(v)$, the minimum degree in graph $G$ by $\delta(G)$, and the neighborhood of $v$ by $N(v)$ respectively. Let $H$ be a subgraph of $G$ and $v$ be a vertex in $H$. We denote the set of vertices that are adjacent to $v$ in $H$ by $N_H(v)$.  Let $x_1, x_2,\dots, x_k$ be $k$ vertices in $H$.   The set of vertices in $H$ which are adjacent to all these $k$ vertices, $x_1, x_2,\dots, x_k$, is denoted by $N^*_H(x_1,x_2,\dots,x_k)$.  For brevity, we may omit the subscript in the notation whenever the graph we are dealing with is clear.  Let $A$ and $B$ be subsets $V(G)$, then the number of edges between them is denoted by $e(A, B)$.  We denote the cycle of length $6$ (or simply a $6$ vertex cycle) by $C_6$ or $6$-cycle. A $7$-wheel, denoted by $W_7$, is a $7$-vertex graph containing a $C_6$ and a vertex that is adjacent to all vertices of the cycle. 
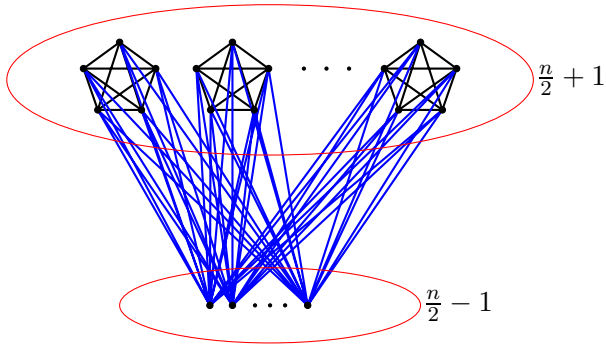
\begin{figure}[h]
\centering
\begin{tikzpicture}[scale=0.1]
\draw[thick](0,5)--(-4.8,1.5)--(-2.9,-4)--(2.9,-4)--(4.8,1.5)--(0,5);
\draw[thick](0,5)--(2.9,-4)(0,5)--(-2.9,-4)(2.9,-4)--(-4.8,1.5)(-2.9,-4)--(4.8,1.5)(4.8,1.5)--(-4.8,1.5);
\draw[thick](15,5)--(10.2,1.5)--(12.1,-4)--(17.9,-4)--(19.8,1.5)--(15,5);
\draw[thick](15,5)--(17.9,-4)(15,5)--(12.1,-4)(17.9,-4)--(10.2,1.5)(12.1,-4)--(19.8,1.5)(19.8,1.5)--(10.2,1.5);
\draw[thick](40,5)--(35.2,1.5)--(37.1,-4)--(42.9,-4)--(44.8,1.5)--(40,5);
\draw[thick](40,5)--(42.9,-4)(40,5)--(37.1,-4)(42.9,-4)--(35.2,1.5)(37.1,-4)--(44.8,1.5)(44.8,1.5)--(35.2,1.5);
\draw[thick,blue](-4.8,1.5)--(25,-30)(-4.8,1.5)--(15,-30)(-4.8,1.5)--(12,-30)(4.8,1.5)--(25,-30)(4.8,1.5)--(15,-30)(4.8,1.5)--(12,-30)(0,5)--(25,-30)(0,5)--(15,-30)(0,5)--(12,-30)(2.9,-4)--(25,-30)(2.9,-4)--(15,-30)(2.9,-4)--(12,-30)(-2.9,-4)--(25,-30)(-2.9,-4)--(15,-30)(-2.9,-4)--(12,-30);
\draw[thick,blue](10.2,1.5)--(25,-30)(10.2,1.5)--(15,-30)(10.2,1.5)--(12,-30)(19.8,1.5)--(25,-30)(19.8,1.5)--(15,-30)(19.8,1.5)--(12,-30)(15,5)--(25,-30)(15,5)--(15,-30)(15,5)--(12,-30)(17.9,-4)--(25,-30)(17.9,-4)--(15,-30)(17.9,-4)--(12,-30)(12.1,-4)--(25,-30)(12.1,-4)--(15,-30)(12.1,-4)--(12,-30);
\draw[thick,blue](35.2,1.5)--(25,-30)(35.2,1.5)--(15,-30)(35.2,1.5)--(12,-30)(40,5)--(25,-30)(40,5)--(15,-30)(40,5)--(12,-30)(44.8,1.5)--(25,-30)(44.8,1.5)--(15,-30)(44.8,1.5)--(12,-30)(42.9,-4)--(25,-30)(42.9,-4)--(15,-30)(42.9,-4)--(12,-30)(37.1,-4)--(25,-30)(37.1,-4)--(15,-30)(37.1,-4)--(12,-30);
\draw[rotate around={0:(20,0)},red] (20,0) ellipse (35 and 10);
\draw[rotate around={0:(20,-30)},red] (20,-30) ellipse (20 and 5);
\draw[fill=black](27.5,1.5)circle(6pt);
\draw[fill=black](30.5,1.5)circle(6pt);
\draw[fill=black](24.5,1.5)circle(6pt);
\draw[fill=black](20,-30)circle(6pt);
\draw[fill=black](22,-30)circle(6pt);
\draw[fill=black](18,-30)circle(6pt);
\draw[fill=black](0,5)circle(12pt);
\draw[fill=black](4.8,1.5)circle(12pt);
\draw[fill=black](-4.8,1.5)circle(12pt);
\draw[fill=black](-2.9,-4)circle(12pt);
\draw[fill=black](2.9,-4)circle(12pt);
\draw[fill=black](10.2,1.5)circle(12pt);
\draw[fill=black](15,5)circle(12pt);
\draw[fill=black](19.8,1.5)circle(12pt);
\draw[fill=black](17.9,-4)circle(12pt);
\draw[fill=black](12.1,-4)circle(12pt);
\draw[fill=black](35.2,1.5)circle(12pt);
\draw[fill=black](40,5)circle(12pt);
\draw[fill=black](44.8,1.5)circle(12pt);
\draw[fill=black](42.9,-4)circle(12pt);
\draw[fill=black](37.1,-4)circle(12pt);
\draw[fill=black](25,-30)circle(12pt);
\draw[fill=black](15,-30)circle(12pt);
\draw[fill=black](12,-30)circle(12pt);
\node at (60,0) {$\frac{n}{2}+1$};
\node at (45,-30) {$\frac{n}{2}-1$};
\end{tikzpicture}
\caption{Extremal construction when $n$ is even and $n\equiv 2(\text{mod }10)$.}
\label{fig1}
\end{figure}
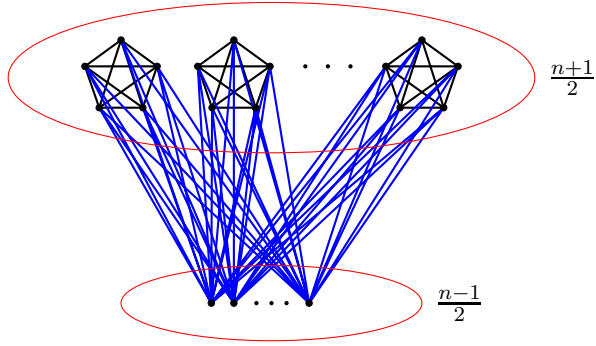
\begin{figure}[h]
\centering
\begin{tikzpicture}[scale=0.1]
\draw[thick](0,5)--(-4.8,1.5)--(-2.9,-4)--(2.9,-4)--(4.8,1.5)--(0,5);
\draw[thick](0,5)--(2.9,-4)(0,5)--(-2.9,-4)(2.9,-4)--(-4.8,1.5)(-2.9,-4)--(4.8,1.5)(4.8,1.5)--(-4.8,1.5);
\draw[thick](15,5)--(10.2,1.5)--(12.1,-4)--(17.9,-4)--(19.8,1.5)--(15,5);
\draw[thick](15,5)--(17.9,-4)(15,5)--(12.1,-4)(17.9,-4)--(10.2,1.5)(12.1,-4)--(19.8,1.5)(19.8,1.5)--(10.2,1.5);
\draw[thick](40,5)--(35.2,1.5)--(37.1,-4)--(42.9,-4)--(44.8,1.5)--(40,5);
\draw[thick](40,5)--(42.9,-4)(40,5)--(37.1,-4)(42.9,-4)--(35.2,1.5)(37.1,-4)--(44.8,1.5)(44.8,1.5)--(35.2,1.5);
\draw[thick,blue](-4.8,1.5)--(25,-30)(-4.8,1.5)--(15,-30)(-4.8,1.5)--(12,-30)(4.8,1.5)--(25,-30)(4.8,1.5)--(15,-30)(4.8,1.5)--(12,-30)(0,5)--(25,-30)(0,5)--(15,-30)(0,5)--(12,-30)(2.9,-4)--(25,-30)(2.9,-4)--(15,-30)(2.9,-4)--(12,-30)(-2.9,-4)--(25,-30)(-2.9,-4)--(15,-30)(-2.9,-4)--(12,-30);
\draw[thick,blue](10.2,1.5)--(25,-30)(10.2,1.5)--(15,-30)(10.2,1.5)--(12,-30)(19.8,1.5)--(25,-30)(19.8,1.5)--(15,-30)(19.8,1.5)--(12,-30)(15,5)--(25,-30)(15,5)--(15,-30)(15,5)--(12,-30)(17.9,-4)--(25,-30)(17.9,-4)--(15,-30)(17.9,-4)--(12,-30)(12.1,-4)--(25,-30)(12.1,-4)--(15,-30)(12.1,-4)--(12,-30);
\draw[thick,blue](35.2,1.5)--(25,-30)(35.2,1.5)--(15,-30)(35.2,1.5)--(12,-30)(40,5)--(25,-30)(40,5)--(15,-30)(40,5)--(12,-30)(44.8,1.5)--(25,-30)(44.8,1.5)--(15,-30)(44.8,1.5)--(12,-30)(42.9,-4)--(25,-30)(42.9,-4)--(15,-30)(42.9,-4)--(12,-30)(37.1,-4)--(25,-30)(37.1,-4)--(15,-30)(37.1,-4)--(12,-30);
\draw[rotate around={0:(20,0)},red] (20,0) ellipse (35 and 10);
\draw[rotate around={0:(20,-30)},red] (20,-30) ellipse (20 and 5);
\draw[fill=black](27.5,1.5)circle(6pt);
\draw[fill=black](30.5,1.5)circle(6pt);
\draw[fill=black](24.5,1.5)circle(6pt);
\draw[fill=black](20,-30)circle(6pt);
\draw[fill=black](22,-30)circle(6pt);
\draw[fill=black](18,-30)circle(6pt);
\draw[fill=black](0,5)circle(12pt);
\draw[fill=black](4.8,1.5)circle(12pt);
\draw[fill=black](-4.8,1.5)circle(12pt);
\draw[fill=black](-2.9,-4)circle(12pt);
\draw[fill=black](2.9,-4)circle(12pt);
\draw[fill=black](10.2,1.5)circle(12pt);
\draw[fill=black](15,5)circle(12pt);
\draw[fill=black](19.8,1.5)circle(12pt);
\draw[fill=black](17.9,-4)circle(12pt);
\draw[fill=black](12.1,-4)circle(12pt);
\draw[fill=black](35.2,1.5)circle(12pt);
\draw[fill=black](40,5)circle(12pt);
\draw[fill=black](44.8,1.5)circle(12pt);
\draw[fill=black](42.9,-4)circle(12pt);
\draw[fill=black](37.1,-4)circle(12pt);
\draw[fill=black](25,-30)circle(12pt);
\draw[fill=black](15,-30)circle(12pt);
\draw[fill=black](12,-30)circle(12pt);
\node at (60,0) {$\frac{n+1}{2}$};
\node at (45,-30) {$\frac{n-1}{2}$};
\end{tikzpicture}
\caption{Extremal construction when $n$ is odd and $n\equiv 1(\text{mod }10)$.}
\label{fig2}
\end{figure}
\begin{figure}[h]
\centering
\begin{tikzpicture}[scale=0.1]
\draw[rotate around={0:(0,0)},red] (0,0) ellipse (28 and 12);
\draw[rotate around={0:(0,-40)},red] (0,-40) ellipse (28 and 12);
\draw[ultra thick](0,5)--(4,0)--(-4,0)--(0,5)(-14,5)--(-10,0)--(-18,0)--(-14,5)(20,5)--(16,0)--(24,0)--(20,5);
\draw[ultra thick](0,-35)--(4,-40)--(-4,-40)--(0,-35)(-14,-35)--(-10,-40)--(-18,-40)--(-14,-35)(20,-35)--(16,-40)--(24,-40)--(20,-35);
\draw[thick,blue](0,5)--(0,-35)(0,5)--(4,-40)(0,5)--(-4,-40)(0,5)--(-14,-35)(0,5)--(-10,-40)(0,5)--(-18,-40)(0,5)--(20,-35)(0,5)--(16,-40)(0,5)--(24,-40);
\draw[thick,blue](4,0)--(0,-35)(4,0)--(4,-40)(4,0)--(-4,-40)(4,0)--(-14,-35)(4,0)--(-10,-40)(4,0)--(-18,-40)(4,0)--(20,-35)(4,0)--(16,-40)(4,0)--(24,-40);
\draw[thick,blue](-4,0)--(0,-35)(-4,0)--(4,-40)(-4,0)--(-4,-40)(-4,0)--(-14,-35)(-4,0)--(-10,-40)(-4,0)--(-18,-40)(-4,0)--(20,-35)(-4,0)--(16,-40)(-4,0)--(24,-40);
\draw[thick,blue](-14,5)--(0,-35)(-14,5)--(4,-40)(-14,5)--(-4,-40)(-14,5)--(-14,-35)(-14,5)--(-10,-40)(-14,5)--(-18,-40)(-14,5)--(20,-35)(-14,5)--(16,-40)(-14,5)--(24,-40);
\draw[thick,blue](-10,0)--(0,-35)(-10,0)--(4,-40)(-10,0)--(-4,-40)(-10,0)--(-14,-35)(-10,0)--(-10,-40)(-10,0)--(-18,-40)(-10,0)--(20,-35)(-10,0)--(16,-40)(-10,0)--(24,-40);
\draw[thick,blue](-18,0)--(0,-35)(-18,0)--(4,-40)(-18,0)--(-4,-40)(-18,0)--(-14,-35)(-18,0)--(-10,-40)(-18,0)--(-18,-40)(-18,0)--(20,-35)(-18,0)--(16,-40)(-18,0)--(24,-40);
\draw[thick,blue](20,5)--(0,-35)(20,5)--(4,-40)(20,5)--(-4,-40)(20,5)--(-14,-35)(20,5)--(-10,-40)(20,5)--(-18,-40)(20,5)--(20,-35)(20,5)--(16,-40)(20,5)--(24,-40);
\draw[thick,blue](16,0)--(0,-35)(16,0)--(4,-40)(16,0)--(-4,-40)(16,0)--(-14,-35)(16,0)--(-10,-40)(16,0)--(-18,-40)(16,0)--(20,-35)(16,0)--(16,-40)(16,0)--(24,-40);
\draw[thick,blue](24,0)--(0,-35)(24,0)--(4,-40)(24,0)--(-4,-40)(24,0)--(-14,-35)(24,0)--(-10,-40)(24,0)--(-18,-40)(24,0)--(20,-35)(24,0)--(16,-40)(24,0)--(24,-40);
\draw[fill=black](0,5)circle(15pt);
\draw[fill=black](4,0)circle(15pt);
\draw[fill=black](-4,0)circle(15pt);
\draw[fill=black](-14,5)circle(15pt);
\draw[fill=black](-10,0)circle(15pt);
\draw[fill=black](-18,0)circle(15pt);
\draw[fill=black](10,0)circle(8pt);
\draw[fill=black](13,0)circle(8pt);
\draw[fill=black](7,0)circle(8pt);
\draw[fill=black](20,5)circle(15pt);
\draw[fill=black](16,0)circle(15pt);
\draw[fill=black](24,0)circle(15pt);
\draw[fill=black](0,-35)circle(15pt);
\draw[fill=black](4,-40)circle(15pt);
\draw[fill=black](-4,-40)circle(15pt);
\draw[fill=black](-14,-35)circle(15pt);
\draw[fill=black](-10,-40)circle(15pt);
\draw[fill=black](-18,-40)circle(15pt);
\draw[fill=black](10,-40)circle(8pt);
\draw[fill=black](13,-40)circle(8pt);
\draw[fill=black](7,-40)circle(8pt);
\draw[fill=black](20,-35)circle(15pt);
\draw[fill=black](16,-40)circle(15pt);
\draw[fill=black](24,-40)circle(15pt);
\end{tikzpicture}
\caption{Extremal construction when $n$ is divisible by 6.}
\label{fig3}
\end{figure}
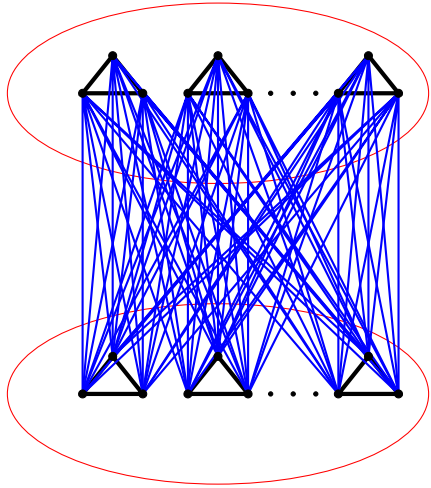

\section{Proof of Theorem \ref{maintheorem}}
We will be using the following  classical stability result of Erd\H os and Simonovits.
\begin{theorem}\cite{EDR3}
Let $k \geq 2$ and suppose that $H$ is a graph with $\chi(H) =k + 1$. If G is an H-free graph with $e(G) \geq t_k(n)-o(n^2)$, then G can be formed from $T_k(n)$ by adding and deleting $o(n^2)$ edges.
\end{theorem}
Since $\chi(TP_3)=3$, the above theorem can be restated as follows.
\begin{theorem}\label{stability}
For every $\gamma>0$, there exists an $\epsilon>0$ and $n_0(\gamma)$ such that for every $n$-vertex, $n>n_0(\gamma)$, and $TP_3$-free graph $G$ such that $e(G)\geq \frac{n^2}{4}-\epsilon n^2$, we have
$$|E(G)\Delta E(T_{2}(n))|\leq \gamma n^2.$$
\end{theorem}

We will prove the following version of Theorem \ref{maintheorem}.
\begin{theorem}\label{mainthm}
For $\delta >0$ and $n\geq \frac{5n_0(\delta)}{2\delta}$, the maximum number of edges in an $n$-vertex $TP_3$-free graph is $\ex(n,TP_3)\leq \frac{n^2}{4}+(1+\delta)n$. 
\end{theorem}

Given a $\delta$, we define the following functions of $\delta$.   The $n_0(\delta)$ in Theorem \ref{mainthm} is coming from the Theorem \ref{stability} and let $\beta(\delta)\geq \frac{\delta}{9296}$.  Whereas $\gamma(\delta)$ satisfies the inequalities $\beta^3 + 512 \beta \gamma^2<16\beta (\beta + 1) (2 \beta + 1) \gamma$ and $\frac{\delta}{1328}\times \frac{\frac{1}{2}-3\beta }{3}<\gamma$.  For brevity of the paper, we do not calculate these functions preciously.

For technical reasons, we start by proving the following weaker version of Theorem \ref{mainthm}.
\begin{lemma}\label{lemma1}
Let $G$ is a $TP_3$-free graph on $n$, $n\geq 10$ vertices. Then $e(G)\leq \frac{n^2}{4}+\frac{7}{2}n$.
\end{lemma}
\begin{proof}
The maximum number of edges in $7$-wheel free graph on $n$ vertices is $\text{ex}(n,W_7)=\lfloor{\frac{n^2}{4}+\frac{n}{2}+1}\rfloor$ \cite{TOMASZ}, which is less than or equal $\frac{n^2}{4}+\frac{7}{2}n$. So, we may assume that $G$ contains a $7$-wheel. We claim that each edge in $G$ is contained in at least $8$ triangles. Suppose not and there is an edge $xy\in E(G)$ such that $|N(x,y)|\leq 7$. In this case, the number of edges that are incident to either $x$ or $y$ is at most $n+6$.  By the induction hypothesis,
$$e(G)\leq e(G-\{x,y\})+(n+6)\leq \frac{(n-2)^2}{4}+\frac{7}{2}(n-2)+(n+6)=\frac{n^2}{4}+\frac{7}{2}n.$$
One can check that the statement also holds for small $n$.

Now consider a $7$-wheel in $G$ with $6$-cycle $x_1x_2x_3x_4x_5x_6x_1$ and center $y$. For any edge $x_ix_j$ in the $6$-cycle, it can be easily seen that there are at least $3$ vertices in $V(G)\backslash \{x_1,x_2,\dots, x_6,y\}$ which are adjacent to both $x_i$ and $x_j$. Therefore by the Pigeonhole principle, we can find three distinct vertices, say $y_1, y_2$ and $y_3$ which are in $N^*(x_1,x_2), N^*(x_3,x_4)$, and $N^*(x_5,x_6)$ respectively.   This is a contradiction as $G$ does not contain a $TP_3$.   
\end{proof}

\begin{lemma}\label{lemma2}
Let $\delta>0$ be given. Let $G$ be an $n$-vertex, $n\geq \frac{5n_0(\gamma)}{2\delta}$ with $e(G)>\frac{n^2}{4}+(1+\delta)n$ edges. Then either $G$ contains a $TP_3$ or $G$ contains a subgraph $G_0$ on $n_0$ vertices such that $e(G_0)> \frac{n_0^2}{4}+(1+\delta)n_0$ with $d(x)>\floor{\frac{n_0}{2}+1}$,
for all $x\in V(G_0)$ and any two adjacent vertices are incident to at least $n_0+2$ common vertices (so each edge is contained in at least three triangles).
\end{lemma}

\begin{proof}
Define a subgraph $H$ of $G$ as good if $e(H)> \frac{v(H)^2}{4}+(1+\delta)v(H)$ with 
\begin{equation}\label{equation2}
d(x)>\floor{\frac{v(H)}{2}+1},
\end{equation}
for all $x\in V(H)$ and any two adjacent vertices are incident to at least $v(H)+2$ edges.

If every vertex in $G$ satisfies the property (\ref{equation2}) (i.e., $G$ itself is good), then the lemma holds.

Otherwise, we delete the vertex in $G$ if it doesn't satisfy the degree condition in (\ref{equation2}) or along with one of its neighbors, they have fewer than $V(G)+2$ edges incident to it. We repeat this step, say $m$ times, till we get a subgraph $H$, satisfying the property (\ref{equation2}). 

We claim the following: 
\begin{claim}$e(H)\geq\frac{(n-m)^2}{4}+(1+\delta)(n-m)+\delta m.$
\end{claim}
\begin{proof}
Suppose not and $e(H)< \frac{(n-m)^2}{4}+(1+\delta)(n-m)+\delta m.$ We distinguish the following four cases based on the parity of $n$ and $m$ to complete the proof.
\subsubsection*{Case 1: $n$ is odd}
The sequence of the number of edges we delete form $G$ in each steps when $m$ is even and $m$ is odd are respectively $$\left(\frac{n+1}{2},\frac{n+1}{2},\frac{n-1}{2},\frac{n+1}{2},\dots, \frac{n-m+3}{2},\frac{n-m+3}{2}\right)$$ and $$\left(\frac{n+1}{2},\frac{n+1}{2},\frac{n-1}{2},\frac{n+1}{2},\dots, \frac{n-m+4}{2},\frac{n-m+4}{2}, \frac{n-m+2}{2}\right).$$ It can be checked that the number of edges be deleted after $m$ steps are respectively $\frac{m}{4}(2n-m+4)$ and $\frac{(m-1)}{4}(2n-m+5)+\frac{n-m+2}{2}=\frac{mn}{2}-\frac{m^2}{4}+m-\frac{1}{4}.$
Thus, when $m$ is even, 
\begin{align*}
e(G)\leq E(H)+\frac{m}{4}(2n-m+4)&<\left(\frac{(n-m)^2}{4}+(1+\delta)(n-m)+\delta m\right)+\frac{m}{4}(2n-m+4)\\&=\frac{n^2}{4}+(1+\delta)n,
\end{align*}
which is a contradiction.
When $m$ is odd, we have 
\begin{align*}
e(G)\leq E(H)+\frac{m}{4}(2n-m+4)&<\left(\frac{(n-m)^2}{4}+(1+\delta)(n-m)+\delta m\right)-\frac{m^2}{4}+\frac{mn}{2}+m-\frac{1}{4}\\&=\frac{n^2}{4}+(1+\delta)n-\frac{1}{4},
\end{align*}
which is again a contradiction.
\subsubsection*{Case 2: $n$ is even}
The sequence of the number of edges deleted in $m$ steps from $G$, when $m$ is odd and $m$ is even, are respectively $$\left(\frac{n+2}{2},\frac{n}{2},\frac{n}{2},\dots,\frac{n-m+3}{2},\frac{n-m+3}{2}\right)$$ and $$\left(\frac{n+2}{2},\frac{n}{2},\frac{n}{2},\dots,\frac{n-m+4}{2},\frac{n-m+4}{2},\frac{n-m+2}{2}\right).$$

Again it can be checked that the number of edges deleted after $m$ steps are respectively $\frac{m-1}{4}(2n-m+3)+\frac{n+2}{2}=-\frac{m^2}{4}+\frac{mn}{2}+m+\frac{1}{4}$ and $\frac{m-2}{4}(2n-m+4)+\frac{n+2}{2}+\frac{n-m+2}{2}=\frac{mn}{2}-\frac{m^2}{4}+m.$ 
When $m$ is even, we have 

\begin{align*}
e(G)\leq E(H)+\frac{m}{4}(2n-m+4)&<\left(\frac{(n-m)^2}{4}+(1+\delta)(n-m)+\delta m\right)-\frac{m^2}{4}+\frac{mn}{2}+m+\frac{1}{4}\\&=\frac{n^2}{4}+(1+\delta)n+\frac{1}{4}.
\end{align*}
Clearly, $e(G)\leq \frac{n^2}{4}+(1+\delta)n$. Otherwise, we get an integer between $\frac{n^2}{4}+(1+\delta)n$ and $\frac{n^2}{4}+(1+\delta)n+\frac{1}{4}$, which is not true. This contradicts the fact that $e(G)>\frac{n^2}{4}+(1+\delta)n$.

When $m$ is odd, we have 
\begin{align*}
e(G)\leq E(H)+\frac{m}{4}(2n-m+4)&<\left(\frac{(n-m)^2}{4}+(1+\delta)(n-m)+\delta m\right)-\frac{m^2}{4}+\frac{mn}{2}+m\\&=\frac{n^2}{4}+(1+\delta)n,
\end{align*}
which is again a contradiction. 
\end{proof}
If $H$ contains a $TP_3$, we are immediately done.  Hence consider $H$ is $TP_3$-free.  By the previous lemma, $e(H)\leq \frac{(n-m)^2}{4}+\frac{7}{2}(n-m)$. Thus, 
\begin{equation*}
\begin{split}
    \frac{(n-m)^2}{4}+(1+\delta)(n-m)+\delta m &\leq \frac{(n-m)^2}{4}+\frac{7}{2}(n-m).
\end{split}
\end{equation*}
Hence, $$ m\leq\frac{2.5-\delta}{2.5}n.$$

This implies $n-m\geq \frac{2\delta n}{5}$. The condition, $n\geq\frac{5n_0(\gamma)}{2\delta}$ implies $n-m\geq n_0(\gamma)$ and thus we found the good subgraph $H$ of $G$.\\
\end{proof}

\begin{remark}
For the rest of the write-up, we always work on this ``good" subgraph and to simplify notations we denote it by $G$.
\end{remark}
\begin{definition}
We call a $7$-wheel in a graph $G$ with the $6$-cycle, say $x_1x_2x_3x_4x_5x_6x_1$, and center $y$, as a \textbf{sparse $7$-wheel}, if $x_ix_{i+2}\notin E(G)$ for all $i\in\{1,2\dots,6\}$ (see Figure \ref{goodwheel}).  
\end{definition}

\begin{figure}[h]
\centering
\begin{tikzpicture}[scale=0.6]
\draw[thick](-2,-2)--(-4,-4)(2,-2)--(4,-4)(2,-2)--(-2,-6)(4,-4)--(2,-6)(-2,-2)--(2,-6)(-4,-4)--(-2,-6);
\draw[thick](-2,-2)--(2,-2)(-4,-4)--(4,-4)(-2,-6)--(2,-6);
\draw[dashed, red](-4,-4)--(2,-2)--(2,-6)--(-4,-4)(4,-4)--(-2,-2)--(-2,-6)--(4,-4);
\draw[fill=black](-2,-2)circle(5pt);
\draw[fill=black](2,-2)circle(5pt);
\draw[fill=black](-4,-4)circle(5pt);
\draw[fill=black](0,-4)circle(5pt);
\draw[fill=black](4,-4)circle(5pt);
\draw[fill=black](-2,-6)circle(5pt);
\draw[fill=black](2,-6)circle(5pt);
\node at (-3,-2) {$x_1$};
\node at (3,-2) {$x_2$};
\node at (4.5,-4) {$x_3$};
\node at (-4.5,-4) {$x_6$};
\node at (2,-7) {$x_4$};
\node at (-2,-7) {$x_5$};
\node at (0,-3.5) {$y$};
\end{tikzpicture}
\caption{A sparse $7$-wheel, the doted red edges are not in $G$.}
\label{goodwheel}
\end{figure}
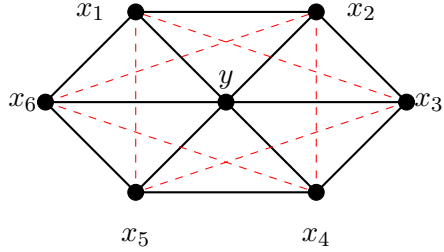
\begin{lemma}\label{lemma4}
Let $\delta>0$ and $G$ be a graph on $n$ vertices containing a sparse $7$-wheel and $e(G)\geq \frac{n^2}{4}+(1+\delta)n$, then $G$ contains a $TP_3$. 
\end{lemma}

\begin{proof}
Suppose $e(G)>\frac{n^2}{4}+(1+\delta)n$. Then by Lemma \ref{lemma2}, $G$ contains a good subgraph $H$. That means,  
\begin{equation}\label{equation1}d(x)>
\begin{cases}
\frac{v(H)}{2}+1, &2\mid v(H),\\
\frac{v(H)+1}{2}, &2\nmid v(H).\\
\end{cases}
\end{equation}
For all $x\in V(H)$ and any two adjacent vertices that are incident to at least $v(H)+2$ edges( and so every edge is contained in at least three triangles).  Note $G$ is a good subgraph.

Let a sparse $7$-wheel in $G$ be with center $y$ and $6$-cycle $x_1x_2x_3x_4x_5x_1$ as shown in Figure \ref{goodwheel}. Since $G$ is good, for each $x_ix_{i+1}$, $i\in \{1,2,\dots,6\}$, $|N(x_i,x_{i+1})|\geq 3$. Moreover, for each $x_ix_{i+1}$, $i\in \{1,2,\dots,6\}$, all the remaining four vertices of the cycle are not in $N(x_i,x_{i+1})$. Indeed, without loss of generality consider the edge $x_1x_2$. $x_3$ and $x_4$ are not in $N(x_1,x_2)$, since $G$ the wheel is sparse and hence they are not in $N(x_1)$ and $N(x_2)$ respectively. With similar argument $x_6$ and $x_5$ are not in $N(x_1,x_2)$.  Therefore, there exist at least two vertices in $V(G)\backslash\{x_1,x_2,\dots,x_6,y\}$, which are in $N(x_i,x_{i+1})$. Take the matching $x_1x_2,x_3x_4$ and $x_5x_6$. If there are three distinct vertices in $V(G)\backslash\{x_1,x_2,\dots, x_6,y\}$, which are in $N(x_1,x_2)\cup N(x_3,x_4)\cup N(x_5,x_6)$, then $TP_3$ in $G$. Indeed, suppose not. Let $z_1,z_2$ and $z_3$ be vertices in $V(G)\backslash\{x_1,x_2,\dots, x_6,y\}$ such that $\{a,b,c\}\subset N(x_1,x_2)\cup N(x_3,x_4)\cup N(x_5,x_6)$. From the property that $G$ is contains no $TP_3$ and $|N(x_1,x_2)|$, $N|(x_3,x_4)|$ and $|N(x_5,x_6)|$ are at least 3, then each of the sets $N(x_1,x_2)$, $N(x_3,x_4)$ and $N(x_5,x_6)$ must contain at least two of the vertices in $\{z_1,z_2,z_3\}$. By the Hall's Theorem, we get distinct pairing of $z_1,z_2,z_3$ and $N(x_1,x_2),N(x_3,x_4)$ and $N(x_5,x_6)$ such that $z_i\in N(x_j,x_k)$, $i\in \{1,2,3\}$ and $(j,k)\in \{(1,2),(3,4),(5,6)\}$, which is a contradiction to the fact that $G$ does not contain $TP_3$. 
Now we may assume that there are only two distinct vertices, say $v_1$ and $v_2$ in  $V(G)\backslash\{x_1,x_2,\dots, x_6,y\}$, such that $N(x_1,x_2,\dots,x_6)=\{v,v_1,v_2\}$(see Figure \ref{wheelstructure}). 

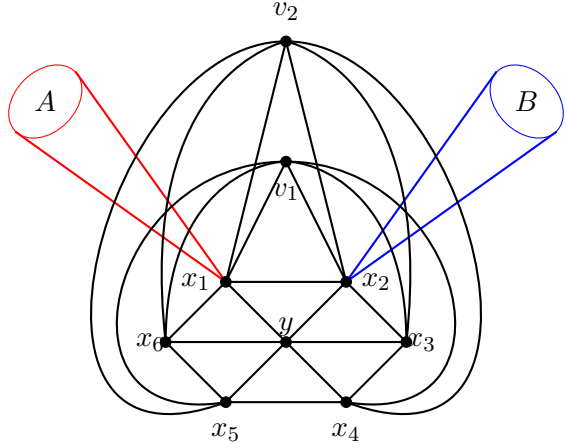
\begin{figure}[h]
\centering
\begin{tikzpicture}[scale=0.4]
\draw[thick](-2,-2)--(-4,-4)(2,-2)--(4,-4)(2,-2)--(-2,-6)(4,-4)--(2,-6)(-2,-2)--(2,-6)(-4,-4)--(-2,-6);
\draw[thick](-2,-2)--(2,-2)(-4,-4)--(4,-4)(-2,-6)--(2,-6);
\draw[thick, red](-2,-2)--(-7,5)(-2,-2)--(-9,3);
\draw[thick,blue](2,-2)--(7,5)(2,-2)--(9,3);
\draw[rotate around={45:(-8,4)},red] (-8,4) ellipse (1.4 and 1);
\draw[rotate around={-45:(8,4)},blue] (8,4) ellipse (1.4 and 1);
\draw[thick](-2,-2)--(0,2)(2,-2)--(0,2)(-2,-2)--(0,6)(2,-2)--(0,6);
\draw[thick](-4,-4)..controls (-4,1) and (-1,2) ..(0,2);
\draw[thick](4,-4)..controls (4,1) and (1,2) ..(0,2);
\draw[thick](-4,-4)..controls (-5,5) and (0,6) ..(0,6);
\draw[thick](4,-4)..controls (5,5) and (0,6) ..(0,6);
\draw[thick](-2,-6)..controls (-8,-7) and (-6,2) ..(0,2);
\draw[thick](2,-6)..controls (8,-7) and (6,2) ..(0,2);
\draw[thick](-2,-6)..controls (-10,-9) and (-6,6) ..(0,6);
\draw[thick](2,-6)..controls (10,-9) and (6,6) ..(0,6);
\node at (-3,-2) {$x_1$};
\node at (3,-2) {$x_2$};
\node at (4.5,-4) {$x_3$};
\node at (-4.5,-4) {$x_6$};
\node at (2,-7) {$x_4$};
\node at (-2,-7) {$x_5$};
\node at (0,-3.5) {$y$};
\node at (-8,4) {$A$};
\node at (8,4) {$B$};
\node at (0,1) {$v_1$};
\node at (0,7) {$v_2$};
\draw[fill=black](0,2)circle(5pt);
\draw[fill=black](0,6)circle(5pt);
\draw[fill=black](-2,-2)circle(5pt);
\draw[fill=black](2,-2)circle(5pt);
\draw[fill=black](-4,-4)circle(5pt);
\draw[fill=black](0,-4)circle(5pt);
\draw[fill=black](4,-4)circle(5pt);
\draw[fill=black](-2,-6)circle(5pt);
\draw[fill=black](2,-6)circle(5pt);
\end{tikzpicture}
\caption{Structure of the subgraph of $G$ with $2$ common neighbors for each vertices on the cycle of the good wheel.}
\label{wheelstructure}
\end{figure}

We prove the lemma for the case when $n$ is odd. With a similar argument, one can also solve the $n$ is even case.

Let $A$ and $B$ be sets of vertices in $V(G)\backslash \{x_1,\dots,x_6,y,v_1,v_2\}$ which are adjacent to $x_1$ and $x_2$ respectively (see Figure \ref{wheelstructure}). Obviously, $A\cap B=\emptyset$. Otherwise, the graph contains a $TP_3$. Thus, either $|A|\leq \frac{n-9}{2} $ or $|B|\leq \frac{n-9}{2}$.

Without loss of generality suppose $|A|\leq \frac{n-9}{2}$. If $|A|\leq \frac{n-11}{2}$, then $d(x_1)\leq |A|+6=\frac{n-11}{2}+6=\frac{n+1}{2}$, which is a contradiction.  

So assume $|A|=\frac{n-9}{2} $. In this case, we also have that $|B|=\frac{n-9}{2}$. We need the following claim to complete proof of the lemma. 
\begin{claim}\label{basic}
Each vertex in $A$ is adjacent to at least one other vertex in $A$ .
\end{claim}
\begin{proof} 
Suppose not and let $x$ be a vertex in $A$ which is adjacent with no other vertex in $A$. The vertex is not adjacent to $x_2$ and $x_6$, otherwise, $G$ contains a $TP_3$.

If $x$ is adjacent to $x_4$, then $x$ is not adjacent to both $x_3$ and $x_5$ too. Otherwise, the graph contains a $TP_3$. In this case, the vertex $x$ is possibly adjacent to $y, v_1, v_2$ and vertices in $B$. Thus considering the vertex $x_1$ which is already adjacent with $x$, we get $d(x)\leq \frac{n-9}{2}+5=\frac{n+1}{2}$.   This is a contradiction to the fact that $G$ is good. 

Let $x$ be adjacent with $x_3$. Then $x$ can not be adjacent to $x_4$. If $x_5$ is not adjacent to $x$, then $d(x)\leq \frac{n-9}{2}+5=\frac{n+1}{2}$, which is a contradiction. So, let $x_5$ be adjacent to $x$. If $x$ is not adjacent to one of the vertices in $\{y,v_1,v_2\}$, then $d(x)\leq \frac{n-9}{2}+5=\frac{n+1}{2}$, which is a contradiction. Otherwise, consider the $7$-wheel, with the $6$-cycle $x_5yx_3v_1x_1v_2x_5$ (see the bold green cycle in Figure \ref{wheelstructureparticular}) and center $x$ . Consider the matching $x_5y$, $x_3v_1$ and $x_1v_2$. We can take the vertices $x_4$, $x_2$ and $x_6$ respectively, which are common neighbors of end vertices of the matching. Thus we get a $TP_3$, in $G$, which is a contradiction to the fact that $G$ is $TP_3$-free.
\end{proof}

\begin{figure}[h]
\centering
\begin{tikzpicture}[scale=0.5]
\draw[thick](-2,-2)--(-4,-4)(2,-2)--(4,-4)(2,-2)--(-2,-6)(4,-4)--(2,-6)(-2,-2)--(2,-6)(-4,-4)--(-2,-6);
\draw[thick](-2,-2)--(2,-2)(-4,-4)--(4,-4)(-2,-6)--(2,-6);
\draw[rotate around={45:(-8,4)},red] (-8,4) ellipse (1.4 and 1);
\draw[thick](-2,-2)--(0,2)(2,-2)--(0,2)(-2,-2)--(0,6)(2,-2)--(0,6);
\draw[thick](-4,-4)..controls (-4,1) and (-1,2) ..(0,2);
\draw[ultra thick, green](4,-4)..controls (4,1) and (1,2) ..(0,2);
\draw[thick](-4,-4)..controls (-5,5) and (0,6) ..(0,6);
\draw[thick](4,-4)..controls (5,5) and (0,6) ..(0,6);
\draw[thick](-2,-6)..controls (-8,-7) and (-6,2) ..(0,2);
\draw[thick](2,-6)..controls (8,-7) and (6,2) ..(0,2);
\draw[ultra thick, green](-2,-6)..controls (-10,-9) and (-6,6) ..(0,6);
\draw[thick](2,-6)..controls (10,-9) and (6,6) ..(0,6);

\draw[ultra thick, green](-2,-6)--(0,-4)--(4,-4)(0,2)--(-2,-2)--(0,6);
\draw[dashed,red](-8,4)--(0,6)(-8,4)--(0,2)(-8,4)--(-2,-2)(-8,4)--(-2,-6)(-8,4)--(4,-4);
\draw[dashed,red](-8,4)..controls (-4, 5) and (0,-2) ..(0,-4);

\node at (-3,-2) {$x_1$};
\node at (3,-2) {$x_2$};
\node at (4.5,-4) {$x_3$};
\node at (-4.5,-4) {$x_6$};
\node at (2,-7) {$x_4$};
\node at (-2,-7) {$x_5$};
\node at (0,-3.5) {$y$};
\node at (-8.5,4) {$x$};
\node at (-8,6) {$A$};
\node at (0,1) {$v_1$};
\node at (0,7) {$v_2$};
\draw[fill=black](-8,4)circle(5pt);
\draw[fill=black](0,2)circle(5pt);
\draw[fill=black](0,6)circle(5pt);
\draw[fill=black](-2,-2)circle(5pt);
\draw[fill=black](2,-2)circle(5pt);
\draw[fill=black](-4,-4)circle(5pt);
\draw[fill=black](0,-4)circle(5pt);
\draw[fill=black](4,-4)circle(5pt);
\draw[fill=black](-2,-6)circle(5pt);
\draw[fill=black](2,-6)circle(5pt);
\end{tikzpicture}
\caption{A graph containing $TP_3$. }
\label{wheelstructureparticular}
\end{figure}
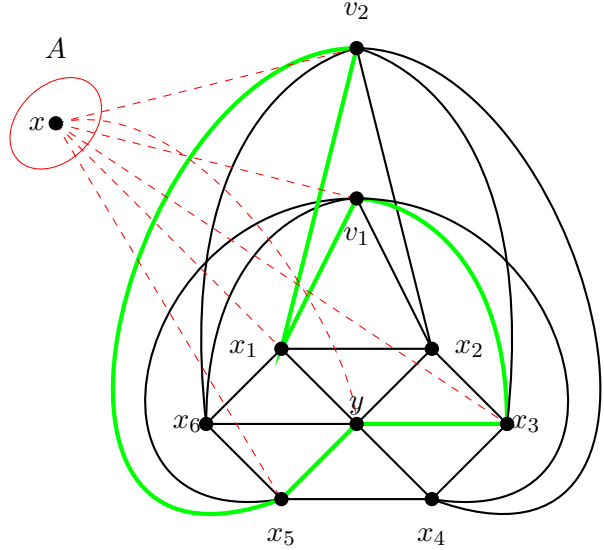 

With the same argument, one can verify that the minimum degree of each vertex in $B$ is at least $1$ in $B$.

Now we finish the proof of Case $1$ of the lemma. Consider the edge $x_5x_6$ and let $A'$ and $B'$ be the set of vertices in $V(G)\backslash \{x_1,\dots,x_6,y,v_1,v_2\}$ which are adjacent to $x_5$ and $x_6$ respectively. For the same reason given above, $|A'|=|B'|=\frac{n-9}{2}$. Clearly $A'\cap B'=\emptyset$. 
Since $A\cap B'=\emptyset$ and $A'\cap B'=\emptyset$, then $|B'\cap B|=|A\cap A'|=\frac{n-9}{2}$.

Let $x\in A\cap A'$.  Suppose $x$ is adjacent to $y$. We can take the $7$-wheel, with $6$-cycle $xx_1x_2x_3x_4x_5x$ and center $y$. By Claim \ref{basic}, there is a vertex $z$ in $A$ which is adjacent to $x$. Since this vertex is adjacent with $x_1$, then taking the matching $xx_1$, $x_2x_3$ and $x_4x_5$ with common neighbors $z, v_1$ and $v_2$ respectively, we show the graph contains a $TP_3$. Therefore, in this case, $x$ cannot be adjacent to $y$. 

Let $t\in B\cap B'$. In this case, $t$ can not be adjacent with $y$. Suppose not. We can take the $7$-wheel, with $6$-cycle $tx_2x_3x_4x_5x_6t$ and center $y$. By Claim \ref{basic}, $t$ is adjacent with a vertex $r$ in $B$. So taking the matching $tx_2$, $x_3x_4$ and $x_5x_6$ with common neighbors $r, v_1$ and $v_2$ respectively, we show that $G$ contains a $TP_3$.  Hence, a contradiction. 

Thus we found that $y$ is a vertex in $G$ with constant degree, which is a contradiction to the fact that $G$ is a good graph. 
\end{proof}

\begin{lemma}\label{lm3}
 Let $G$ be a graph on $n$ vertices, where $n\geq \frac{5n_0(\gamma)}{2\delta}$, and then $e(G)\geq\frac{n^2}{4}+(1+\delta)n$. Let $A$ and $B$ a be partition of $V(G)$ with size as equal as possible and with maximum $e(A,B)$. If $A$ contains (similarly $B$ contains) a vertex, say $x$, such that $d_A(x)\geq \beta n$,   then $G$ contains a $TP_3$. 
\end{lemma}

\begin{proof}
Without loss of generality, suppose there exists vertex $x\in A$ such that $d_A(x)\geq \beta n$.  Obviously $e(G)> \frac{n^2}{4}-\epsilon n^2$, for any $\epsilon>0$. Thus by the stability theorem, $|E(G)\Delta E(T_{n,2})|\leq \gamma n^2$.


Let $A_x$ be the graph induced by the vertices $N_A(x)\cup \{x\}$ in $A$. Hence, we have $e(A_x)\leq \gamma n^2$, which results in $\sum\limits_{y\in V(A_x)}d_{A_x}(y)\leq 2\gamma n^2$. The average degree of $A_x$ is

$$\Bar{d}(A_x)\leq \frac{\sum\limits_{y\in V(A_x)}d_{A_x}(y)}{v(A_x)}\leq \frac{2\gamma n^2}{\beta n}=\frac{2\gamma n}{\beta}.$$

Let $X$ be the set of vertices in $A_x$ with degree at least $\frac{4\gamma n}{\beta}$. It can be checked that the size of $X$ is at most $\frac{\beta n}{2}$. Let $Y=V(A_x)-X$. Thus, $|Y|\geq \frac{\beta n}{2}$ and for each $y\in Y$, $d_Y(y)\leq \frac{4\gamma n}{\beta}$. Now we can color $G[Y]$ with $ \frac{4\gamma n}{\beta}$ colors. The average size of the color class in $G[Y]$ is at least $\frac{\left(\beta n\right)/2}{\left(4\gamma n\right)/\left(\beta \right)}=\frac{\beta^2}{8\gamma}\geq 3.$
Thus we obtained at least $\frac{n}{3}\left(\frac{\beta}{2}-\frac{8\gamma}{\beta}\right)$ induced $K_{1,3}$'s in $A_x$(see Figure \ref{sparswheel}.)

Notice that the graph induced by $B$, denoted by $G_B$, contains at most $\gamma n^2$ edges.  The average degree is $\Bar{d}(G_B)\leq 2\gamma n$. With the same argument as given above, we can keep an overwhelming majority of vertices in $B$ whose degree is at most $4\gamma n$. Indeed, deleting vertices in $B$ whose degree is at least $4\gamma n$, we are left with at least $\frac{n}{4}$ vertices. Let $Z$ be the set of vertices remaining in $B$ after deleting the vertices. We color $G[Z]$ with $4\gamma n$ colors. The average size of the color class in $G[Z]$ is at least $\frac{n/2}{4\gamma n}$. This implies that we can find at least $\frac{1}{3}\left(\frac{n}{4}-2\times 4\gamma n\right)=\frac{n}{3}\left(\frac{1}{4}-8\gamma\right)$ induced triples in $G_B$ (see Figure \ref{sparswheel}.)
\begin{figure}[h]
\centering
\begin{tikzpicture}[scale=0.2]
\draw[thick](0,0)--(0,-5)(0,0)--(-2,-5)(0,0)--(2,-5)(0,0)--(-6,-5)(0,0)--(-8,-5)(0,0)--(-10,-5) (0,0)--(10,-5)(0,0)--(12,-5)(0,0)--(14,-5);
\draw[thick, blue](0,-5)--(10,-25)--(-2,-5)--(12,-25)--(2,-5)--(14,-25)--(0,-5)(10,-25)--(0,0)--(14,-25)(0,0)--(12,-25)(-2,-5)--(14,-25)(0,-5)--(12,-25)(2,-5)--(10,-25);
\draw[dashed, green](-12,-8)--(-12,-2)--(16,-2)--(16,-8)--(-12,-8);
\draw[dashed, green](-12,-28)--(-12,-22)--(16,-22)--(16,-28)--(-12,-28);
\draw[rotate around={0:(0,-5)},red] (0,-5) ellipse (18 and 7);
\draw[rotate around={0:(0,-5)},blue] (0,-5) ellipse (3.5 and 1);
\draw[rotate around={0:(-8,-5)},blue] (-8,-5) ellipse (3.5 and 1);
\draw[rotate around={0:(12,-5)},blue] (12,-5) ellipse (3.5 and 1);
\draw[rotate around={0:(0,-25)},red] (0,-25) ellipse (18 and 7);
\draw[rotate around={0:(0,-25)},blue] (0,-25) ellipse (3.5 and 1);
\draw[rotate around={0:(-8,-25)},blue] (-8,-25) ellipse (3.5 and 1);
\draw[rotate around={0:(12,-25)},blue] (12,-25) ellipse (3.5 and 1);
\draw[fill=black](0,0)circle(10pt);
\draw[fill=black](0,-5)circle(10pt);
\draw[fill=black](-2,-5)circle(10pt);
\draw[fill=black](2,-5)circle(10pt);
\draw[fill=black](-6,-5)circle(10pt);
\draw[fill=black](-8,-5)circle(10pt);
\draw[fill=black](-10,-5)circle(10pt);
\draw[fill=black](6,-5)circle(3pt);
\draw[fill=black](7,-5)circle(3pt);
\draw[fill=black](5,-5)circle(3pt);
\draw[fill=black](10,-5)circle(10pt);
\draw[fill=black](12,-5)circle(10pt);
\draw[fill=black](14,-5)circle(10pt);
\draw[fill=black](0,-25)circle(10pt);
\draw[fill=black](-2,-25)circle(10pt);
\draw[fill=black](2,-25)circle(10pt);
\draw[fill=black](-6,-25)circle(10pt);
\draw[fill=black](-8,-25)circle(10pt);
\draw[fill=black](-10,-25)circle(10pt);
\draw[fill=black](6,-25)circle(3pt);
\draw[fill=black](7,-25)circle(3pt);
\draw[fill=black](5,-25)circle(3pt);
\draw[fill=black](10,-25)circle(10pt);
\draw[fill=black](12,-25)circle(10pt);
\draw[fill=black](14,-25)circle(10pt);
\node at (20,-5) {$A$};
\node at (20,-25) {$B$};
\node[green] at (-14,-25) {$Z$};
\node[green] at (-14,-5) {$Y$};
\end{tikzpicture}
\caption{A sparse $7$-wheel.}
\label{sparswheel}
\end{figure}
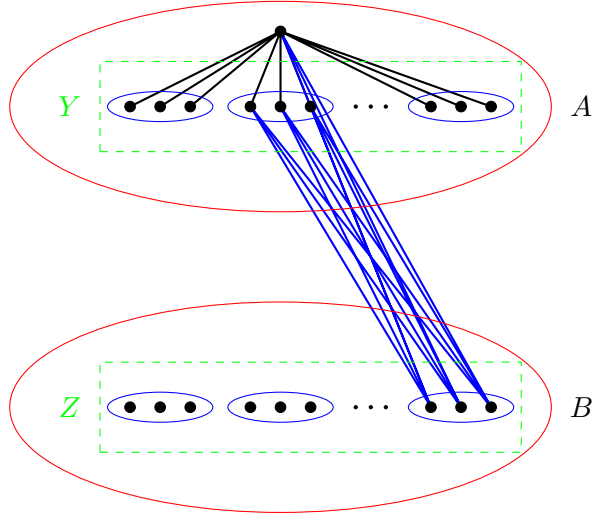

If for each pair of induced $K_{1,3}$ and induced triples obtained in $A$ and $B$ respectively, there is a missing edge, then the number of missed edges is at least $\frac{n}{3}\left(\frac{\beta}{2}-\frac{8\gamma}{\beta}\right)\times \frac{n}{3}\left(\frac{1}{4}-8\gamma\right)$. However if this is greater than $\gamma n^2$, it is a contradiction.  Hence we need the following in-equation to be true:
\begin{equation}\label{bound1}
    \frac{n}{3}\left(\frac{\beta}{2}-\frac{8\gamma}{\beta}\right)\times \frac{n}{3}\left(\frac{1}{4}-8\gamma\right)<\gamma n^2.
\end{equation}
It follows from the definition of $\beta$ and $\gamma$.  Thus there must be an induced $K_{1,3}$ in $A$, which is joined completely to an induced triple of vertices in $B$. Therefore, we get a sparse $7$-wheel. Therefore, $G$ contains a $TP_3$ by Lemma \ref{lemma4}.
\end{proof}

\begin{corollary}
Let $G$ be a graph on $n$ vertices, where $n\geq \frac{5n_0(\gamma)}{2\delta}$, and $e(G)> \frac{n^2}{4}+(1+\delta)n$.  Let $A$ and $B$ be a partition of $V(G)$ with size as equal as possible and with maximum $e(A,B)$.  If $A$ or $B$ has a spider graph as a subgraph, then $G$ contains $TP_3$ as a subgraph.
\end{corollary}
\begin{proof}
Let $S$ denote the spider graph as denoted in Figure \ref{spider2}. Without loss of generality, Suppose $S\subseteq G[A]$.
\begin{figure}[h]
\centering
\begin{tikzpicture}[scale=0.2]
\draw[thick](0,0)--(5,-5)--(5,-10)(0,0)--(0,-5)--(0,-10)(0,0)--(-5,-5)--(-5,-10);
\draw[fill=black](0,0)circle(12pt);
\draw[fill=black](5,-5)circle(12pt);
\draw[fill=black](5,-10)circle(12pt);
\draw[fill=black](0,-5)circle(12pt);
\draw[fill=black](0,-10)circle(12pt);
\draw[fill=black](-5,-5)circle(12pt);
\draw[fill=black](-5,-10)circle(12pt);
\node at (0,1.5) {$x$};
\node at (7,-5) {$w_1$};
\node at (7,-10) {$w_2$};
\node at (2,-5) {$v_1$};
\node at (2,-10) {$v_2$};
\node at (-7,-5) {$u_1$};
\node at (-7,-10) {$u_2$};
\end{tikzpicture}
\caption{A spider graph with three legs and one joint.}
\label{spider2}
\end{figure}
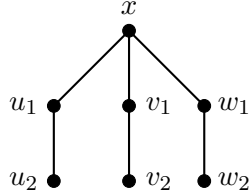

We consider $4$-vertex subsets of $S$, namely $\{x,u_1,u_2,v_1\},\{x,v_1,v_2,w_1\}$ and $\{x,w_1,w_2,u_1\}$.  Note that, if we can find $3$ distinct vertices in $B$ such that, one of them is connected to all the vertices in the above subsets, we immediately find a $TP_3$.  Without loss of generality, assume that the $4$-set $\{x,u_1,u_2,v_1\}$ does not have a common vertex in $B$.  In other words, for every vertex $y\in B$, $y$ is not adjacent to at least one of the vertices in $\{x,u_1,u_2,v_1\}$.  Note that, the average degree of vertices in $\{x,u_1,u_2,v_1\}$ is $\frac{3n}{8}$.  So there exists a vertex $z\in \{x,u_1,u_2,v_1\}$, such that $d_{B}(z)\leq\frac{3n}{8}$.  The minimum degree of the vertices in $G$ is at least $\frac{n}{2}$, thus $d_{A}(z)\geq \frac{n}{8}$.

So we have this large degree vertex in $A$ and are done by the Lemma \ref{lm3}.

\begin{claim}
Given a graph $G_k$ on $k$ vertices, with $2k$ edges.  We can find an independent set of vertices with size $\frac{3k}{55}$.
\end{claim}
\begin{proof}
Say we delete vertices with degrees greater than $10$. Denote the remaining graph  with $G'$.  The number of vertices deleted is denoted by $l$.  The sum of the degrees is at least $10l$.  Thus the number of edges deleted is at least $5l$.  We already know the number of edges in the graph is $2k$, hence $l\leq \frac{2k}{5}$.  Then in $G'$, every vertex has degree at most $10$.  Start by choosing an arbitrary vertex $x\in G'$, delete its neighbors, and continue choosing another vertex in the graph $G'\setminus N(x)$.  With this recursive procedure, we can get an independent set of size $\frac{3k}{55}$. 
\end{proof}

\begin{claim}
  Let $G$ be a graph on $n$ vertices, where $n\geq \frac{5n_0(\gamma)}{2\delta}$. Let $A$ and $B$ a be partition of $V(G)$ with size as equal as possible and with maximum $e(A,B)$. Let $e(A)\geq \frac{n}{2}+\delta \frac{n}{2}$, then the total number of triples of vertices we can find such that they are in $K_{1,3}$'s or induced $k_{1,3}$'s (which are a subgraph of a huge star, with center vertex having degree at-least $84$) is at-least $\frac{\delta n}{664}$.
\end{claim}
\begin{proof}
The degree sum of vertices in $A$ is greater than or equal to $2(\frac{n}{2}+\delta\frac{n}{2})$.  Hence we have vertices that have degree at least $2$.

Let $v$ be a vertex in $A$ such that $d_A(v)=\Delta$.  Let $A_v$ be the graph induced by the vertices $\{v\}\cup N_A(v)$. Note, $A_v$ doesn't contain the spider graph as a subgraph.  We consider the following cases:\\
\textbf{Case $1$: $\Delta\leq 83$.}

Let $x_1,x_2$ and $x_3$ be in $N(v)$.  The vertices $v,x_1,x_2$ and $x_3$ form a $K_{1,3}$.  On deletion of these $4$ vertices, we have deleted at most $332$ edges. Note that $332$ is negligible compared to the number of extra edges in $A$, which was $\delta \frac{n}{2}$. Hence the number of $K_{1,3}$'s we can find is at least $\frac{\delta n}{664}$.\\
\textbf{Case $2$:  $\Delta>84$.}

Denote the vertices in $N(v)$ with $x_i$.  Note that we do not have $3$ independent edges going out of $G_A(v)$ from $x_i$'s, as we have a spider-free graph.  Let $x_1,x_2$, and $x_3$ be vertices degree greater than $2$.  Then by Halls Theorem, we immediately get a matching and $3$ independent edges going from the set $G_A(v)$ to $A\setminus G_A(v)$.  Thus we have at-most $2$ vertices in the set $\{x_i\}$, who have degree greater than $2$.  Thus the number of edges incident to $G_A(v)$ is at most $2(\Delta-1)+2(\Delta-2)+2\Delta\leq 6\Delta.$  

By the previous lemma, in the graph induced by the set of vertices $x_i$, we can find an independent set of size at least $\frac{3\Delta}{55}$.  Hence we can find at least $\frac{\Delta}{55}$ triples such that it forms an induced $K_{1,3}$ with $v$ being the center. The number of $K_{1,3}$'s we can find is at least $\frac{\delta n}{660}$.
\end{proof}

We want to prove $\ex(n,TP_3)\leq \frac{1}{4}n^2+(1+\delta)n$.  Assume that there is a $TP_3$-free graph that has more than $\frac{1}{4}n^2+(1+\delta)n$ edges.  Then one of the bi-partitions has to have more than $\frac{n}{2}+\frac{\delta n}{2}$ edges.  In the next lemma, we show that this is not possible.
\begin{lemma}
Let $G$ be a graph on $n$ vertices, where $n\geq \frac{5n_0(\gamma)}{2\delta}$. Let $A$ and $B$ be partition of $V(G)$ with size as equal as possible and with maximum $e(A,B)$.  Assume that, neither $A$ nor $B$ contains a spider graph as a subgraph and the maximum degree of vertices inside each of the class is $\beta n$. Say $e(A)\geq\frac{n}{2}+\frac{\delta n}{2}$, then $G$ contains a $TP_3$.
\end{lemma}\begin{proof}

By the previous lemma, we have the total number of triples either in $K_{1,3}$'s or induced $K_{1,3}$'s (which are a subgraph of a star, with the center vertex of degree at least $84$) is $\frac{\delta n}{664}$.  Let us consider two cases:
\subsection*{Case $1$:  Half of the triples lie in disjoint $K_{1,3}$'s.}
Consider a vertex $x\in B$.  We know that the maximum degree on $x$ inside $B$ is less than equal to $\beta n$.  So $x$ has at most $\beta n$ non-neighbors in $A$. Thus are at-least $\frac{\delta n}{1328}-\beta n$ triples in disjoint $K_{1,3}$, such that all four of the vertices in the $K_{1,3}$ are adjacent to $x$.  Consider three independent edges in $B$, namely $y_1z_1,y_2z_2$ and $y_3z_3$.  For each of these $6$ vertices, we can find at least $\frac{\delta n}{1328}-\beta n$ triples in disjoint $K_{1,3}$, such that the vertices of the $K_{1,3}$ are joined completely to the given vertex.   Then each of the vertices $y_i$ (similarly $z_i$) is completely connected to all the vertices of at least $\frac{6}{7}$ triples of disjoint $K_{1,3}$ in $A$.  In other words, we need the following in-equation to be true.
\begin{equation}\label{bound2}
    \frac{\delta n}{1328}-\beta n\geq \frac{6}{7}\times\frac{\delta n}{1328}.
\end{equation}
This holds by the definition of $\beta$.  Thus by the Pigeon-hole principle, we have a common triple, such that these $3$ independent edges are connected to it completely.  Denote the vertices of this triple as $x_1,x_2$ and $x_3$.  The vertices $x_1,y_1,x_2y_2$ and $x_3y_3$ along with $x$ form a $7$-wheel.  The triangles $x_1y_1z_1,x_2y_2z_2$, and $x_3y_3z_3$ sitting on the $7$-wheel form a $TP_3$.
\subsection*{Case $2$: Half of the triples lie in induced $K_{1,3}$'s.}
  Let the number of induced $K_{1,3}$'s in each of these stars be $k_i$.  Note that, summing $k_i$ over all the vertices in $A$ which have degree at least $84$, is at least $\frac{\delta n}{1328}$.  Consider the center of one such star in $A$, say $x$. The maximum degree of $x$ in $A$ is less than equal to $\beta n$.  Hence $x$ can have at most $\beta n$ non-neighbors in $B$.  Delete these vertices in $B$ and denote the graph remaining with $B'$. We know that $\Delta(B')\leq \beta n$.    Hence we can color it with $\beta n$ colors and each color class is of size at most $\frac{1}{2\beta}-1$.  Hence we can choose $\frac{\frac{n}{2}-3\beta n}{3}$ independent triples.  Each of these triples must have a missing edge to the root vertices in the $K_{1,3}$ chosen in $A$, otherwise, we are done.  Hence the number of missing edges is equal to $k_1\times \frac{\frac{n}{2}-3\beta n}{3}$.  Summing this over vertices in $A$ with degree at least $25$, we get $\frac{\delta n}{1328}\times \frac{\frac{n}{2}-3\beta n}{3}$.  This can't be bigger than the possible number of missing edges $\gamma n^2$.
  This gives us the following in-equation
  \begin{equation}\label{bound3}
      \frac{\delta n}{1328}\times \frac{\frac{n}{2}-3\beta n}{3}<\gamma n^2,
  \end{equation}
  which holds by definition.  Hence we find a sparse $7$-wheel and we are done.
\end{proof}
\end{proof}

\section{Concluding remarks and conjectures}
Following the two constructions given in Figure \ref{fig1} and Figure \ref{fig2}, we pose the following conjecture concerning $\text{ex}(n, TP_3)$.  
\begin{conjecture}\label{con3}
\begin{equation*}\text{ex}(n,TP_3)\leq
\begin{cases}
\frac{1}{4}n^2+n+1, &\text{if $n$ is even,}\\
\frac{1}{4}n^2+n+\frac{3}{4},&\text{ otherwise}.
\end{cases}
\end{equation*}
\end{conjecture}
We also pose the following conjecture related to $\text{ex}(n,TP_4)$. 
\begin{conjecture}\label{con1}
For $n$ sufficiently large, $\text{ex}(n,TP_4)=\frac{n^2}{4}+\Theta(n^{4/3})$.
\end{conjecture}

To show the lower bound, we consider an $n$-vertex graph $G$  obtained from a complete bipartite graph with color classes as equal as possible and adding a bipartite $C_6$-free graph with $cn^{4/3}$ edges in one of the color classes. Thus, $e(G)\geq \frac{n^2}{4}+O(n^{4/3})$.  The only thing we need to show is $G$ does not contain a $TP_4$. We need the following claim to show that.
\begin{claim}\label{yesclaim}
Every $2$-coloring of the $TP_4$ such that color 1 is independent, contains either a $C_3$ or a $C_6$ in color 2.
\end{claim}
\begin{proof}
Consider a $2$-coloring $c$ of a $TP_4$ such that color $1$ is independent. We want to show that there is either a $C_3$ or a $C_6$ in color $2$. Suppose there is no such $C_3$. Then one of the vertices of the triangle $x_1x_2x_3$ (see Figure \ref{qwxs}) is in color $2$. Without loss of generality, let the color of $x_1$ be $1$. Since $c$ is a $2$-coloring with the property that color $1$ is independent, then all the $6$ neighboring vertices of $x_1$ must be of color $2$. Therefore, we obtain a $C_6$ with color $2$ and this completes the proof. 
\begin{figure}[h]
\centering
\begin{tikzpicture}[scale=0.4]
\draw[fill=black](0,0)circle(7pt);
\draw[fill=black](-2,-2)circle(7pt);
\draw[fill=black](2,-2)circle(7pt);
\draw[fill=black](-4,-4)circle(7pt);
\draw[fill=black](0,-4)circle(7pt);
\draw[fill=black](4,-4)circle(7pt);
\draw[fill=black](-6,-6)circle(7pt);
\draw[fill=black](-2,-6)circle(7pt);
\draw[fill=black](2,-6)circle(7pt);
\draw[fill=black](6,-6)circle(7pt);
\draw[fill=black](-8,-8)circle(7pt);
\draw[fill=black](-4,-8)circle(7pt);
\draw[fill=black](0,-8)circle(7pt);
\draw[fill=black](4,-8)circle(7pt);
\draw[fill=black](8,-8)circle(7pt);
\draw[thick](-2,-2)--(2,-2)(-4,-4)--(4,-4)(-6,-6)--(6,-6)(-8,-8)--(8,-8);
\draw[thick](0,0)--(-8,-8)(2,-2)--(-4,-8)(4,-4)--(0,-8)(6,-6)--(4,-8);
\draw[thick](0,0)--(8,-8)(-2,-2)--(4,-8)(-4,-4)--(0,-8)(-6,-6)--(-4,-8);
\node at (0,-3) {$x_2$};
\node at (-3,-6.5) {$x_1$};
\node at (3,-6.5) {$x_3$};
\end{tikzpicture}
\caption{$TP_4$.}
\label{qwxs}
\end{figure}   
\end{proof}
The following lemma is a consequence of Claim \ref{yesclaim} and hence the lower bound of Conjecture \ref{con1} holds.
\begin{lemma}\label{lemma9}
Let $G$ be a graph obtained from a complete bipartite graph $K_{\frac{n}{2},\frac{n}{2}}$ (with color class $1$ and $2$) and a bipartite, $C_6$-free graph to the color class $2$. Then $G$ is a $TP_4$-free graph.   
\end{lemma}


\section*{Acknowledgments}
 The research of first and second  authors was supported by the National Research, Development and Innovation Office NKFIH, grant  K116769, and the  research of second, third, fourth and fifth authors was supported by the National Research, Development and Innovation Office NKFIH, grant K132696.


\end{document}